\documentclass[reqno, 12pt]{amsart}

\pdfoutput=1

\usepackage{enumerate}
\usepackage{latexsym}
\usepackage[centertags]{amsmath}
\usepackage{amsfonts}
\usepackage{amssymb}
\usepackage{amsthm}
\usepackage{newlfont}
\usepackage{graphics}
\usepackage{color}
\usepackage{float}
\usepackage{diagbox}
\usepackage{hyperref}
\usepackage{longtable}
\usepackage{rotating}
\usepackage{multirow}

\usepackage{extarrows}

\usepackage[sort,compress,numbers]{natbib}

\usepackage[utf8]{inputenc}

\newtheorem{thm}{Theorem}[section]
\newtheorem{cor}[thm]{Corollary}
\newtheorem{lem}[thm]{Lemma}
\newtheorem{prop}[thm]{Proposition}

\theoremstyle{mydefinition}
\newtheorem{dfn}[thm]{Definition}
\theoremstyle{myremark}
\newtheorem{rem}[thm]{Remark}

\allowdisplaybreaks[4]

\def\Z{\mathbb{Z}}

\newcommand{\Q}{{\mathbb{Q}}}

\def\CT{\mathop{\mathrm{CT}}}

\title{Generating functions for the quotients of numerical semigroups}

\author[Feihu Liu]{Feihu Liu$^{1}$}

\address{$^{1}$School of Mathematical Sciences, Capital Normal University,
 Beijing 100048,  PR China}
\email{$^1$\texttt{liufeihu7476@163.com}}
\date{December 18, 2023}
\begin{document}
\maketitle

\begin{abstract}
We propose a class of generating functions denoted by $\textrm{RGF}_p(x)$, which is related to the Sylvester denumerant for the quotients of numerical semigroups. Using MacMahon's partition analysis, we can obtain $\textrm{RGF}_p(x)$ by extracting the constant term of a rational function. We use $\textrm{RGF}_p(x)$ to give a system of generators of the quotient of the numerical semigroup $\langle a_1,a_2,a_3\rangle$ by $p$ for a small positive integer $p$ and we characterise the generators for $\frac{\langle A\rangle}{p}$ for a general numerical semigroup $A$ and any positive integer $p$.
\end{abstract}

\def\D{{\mathcal{D}}}

\noindent
\begin{small}
\emph{Mathematic subject classification}: Primary 11D07; Secondary 05A15, 11D04, 20M13.
\end{small}

\noindent
\begin{small}
\end{small}

\section{Introduction}

Throughout this paper, $\mathbb{Z}$, $\mathbb{N}$, and $\mathbb{Z}^{+}$ denote the set of all integers, non-negative integers, and positive integers, respectively.

A subset $S$ of $\mathbb{N}$ is a \emph{numerical semigroup} if $0\in S$, $\mathbb{N}\setminus S$ is finite and $S$ is closed under the addition in $\mathbb{N}$.
Given a positive integer sequence $A=(a_1,a_2,...,a_k)$, if $\gcd(A)=1$, then
$$\langle A\rangle=\left\{x_1a_1+x_2a_2+\cdots +x_ka_k \mid k\geq 2, x_i\in \mathbb{N}, 1\leq i\leq k \right\}$$
is a numerical semigroup (see \cite{J.C.Rosales}). And, we say that $A$ is a \emph{system of generators} of $S=\langle A\rangle$. If no proper subsequence (or subset) of $A$ generates $S$, then we say that $A$ is a \emph{minimal system of generators} of $S$.
For the above numerical semigroup $\langle A\rangle$, Sylvester \cite{J.J.Sylvester57} defined the function $d(a_0;a_1,a_2,...,a_k)$, called the denumerant, as
$$d(a_0;a_1,a_2,...,a_k)=\#\{(x_1,...,x_k)\mid x_1a_1+x_2a_2+\cdots +x_ka_k=a_0,\ \ x_i\in \mathbb{N} \}.$$
That is, the number of representations of $a_0$ by nonnegative integer linear combinations of $a_1,a_2,...,a_k$.
If $\gcd(A)=1$, then there exists a positive integer $N$ such that $d(a_0;a_1,...,a_k)>0$ for any integer $a_0\geq N$ (see, e.g., \cite[Theorem 1.0.1]{Ramrez Alfonsn}). The \emph{Frobenius number} of $A$ is defined by
$$ F(A)=\max\{a_0\in \mathbb{Z}^{+}\mid d(a_0;a_1,a_2,...,a_k)=0\}.$$
In other words, $F(A)$ is the greatest integer not belonging to $\langle A\rangle$. For more descriptions and results about numerical semigroups, see \cite{Ramrez Alfonsn,A.Assi,J.C.Rosales}.

Suppose $\langle A\rangle$ is a numerical semigroup and $p\in \mathbb{Z}^{+}$. In \cite{Rosales03}, the numerical semigroup
$$\frac{\langle A\rangle}{p}=\{n\in \mathbb{N} \mid pn\in \langle A\rangle\}=\{n \mid pn=x_1a_1+x_2a_2+\cdots +x_ka_k, x_i\in\mathbb{N}, 1\leq i\leq k\}$$
called the \emph{quotient of $\langle A\rangle$ by $p$} was introduced. It is easy to verify that $\frac{\langle A\rangle}{p}$ is a numerical semigroup, that $\langle A\rangle \subseteq\frac{\langle A\rangle}{p}$, and that $\frac{\langle A\rangle}{p}=\mathbb{N}$ if and only if $p\in\langle A\rangle$. For example, let $p=3$ and  $\langle A\rangle=\langle 5,6\rangle=\{0,5,6,10,11,12,15,16,17,18,20\rightarrow\}$, where the symbol $\rightarrow$ means to include all subsequent integers. Then
$\frac{\langle A\rangle}{3}=\{0,2,4\rightarrow\}=\langle 2,5\rangle$.

Let $a_1,a_2,p$ be pairwise relatively prime positive integers. Rosales \cite{Rosales05} obtained a system of generators for $\frac{\langle a_1,a_2\rangle}{2}$ and Rosles and Urbano-Blanco \cite{Rosales2006} gave a characterisation of a system of generators for $\frac{\langle a_1,a_2\rangle}{p}$ by
 means of modular permutations and certain congruence equations. In \cite{E.Cabanillas}, E. Cabanillas discussed the minimal generators of $\frac{\langle a_1,a_2\rangle}{p}$. In \cite{Alessio19}, A. Moscariello also gave a depiction of the generating system of $\frac{\langle A\rangle}{p}$ by defining a class of partitions. There are many open problems related to $\frac{\langle A\rangle}{p}$ (see, e.g., \cite{MDelgado13}).

For $n\in \frac{\langle A\rangle}{p}$, we consider the generating function
$$\textrm{RGF}_p(x)=\sum_{n\geq 0}d(pn;a_1,a_2,...,a_k)x^n,$$
which is referred to as the \emph{representation generating function} of $\frac{\langle A\rangle}{p}$.
The function $\textrm{RGF}_p(x)$ is easily seen to be rational. It is used to obtain a system of generators for $\frac{\langle A\rangle}{p}$.
For example, let $a_1$ and $a_2$ be relatively prime odd positive integers. Then we have
$$\sum_{n\geq 0}d(n;a_1,a_2)x^n=\frac{1}{(1-x^{a_1})(1-x^{a_2})}.$$
For $p=2$, the representation generating function is determined by
$$\textrm{RGF}_2(x^2)=\frac{1}{2}\left(\frac{1}{(1-x^{a_1})(1-x^{a_2})}+\frac{1}{(1-(-x)^{a_1})(1-(-x)^{a_2})}\right)
=\frac{1+x^{a_1+a_2}}{(1-x^{2a_1})(1-x^{2a_2})}.$$
Therefore, we have $\frac{\langle a_1,a_2\rangle}{2}=\langle a_1,a_2,\frac{a_1+a_2}{2}\rangle$. We take several cases in \cite{Rosales05,Rosales2006,Alessio19} as examples and recompute the corresponding generators to check the validity of this idea.

In this paper, we use the idea of \emph{MacMahon's Partition Analysis} \cite{MacMahonCA}
 to represent $\textrm{RGF}_p(x)$ as the constant term of a rational function in a new variable $\lambda$.
 For some small $p\in\mathbb{Z}^{+}$ and $A=(a_1,a_2,a_3)$, we can calculate $\textrm{RGF}_p(x)$ and further obtain a system of generators of the quotient of the numerical semigroup $\langle A\rangle$ by $p$.
Let $a_i=pk_i+t_i$, $0\leq t_i\leq p-1$, $p,k_i\in\mathbb{Z}^{+}$ for $1\leq i\leq 3$ and $\gcd(A)=1$. We obtain the following results in Table \ref{tab-GOr}.
\begin{table}[htbp]
    	\centering
    	\caption{A system of generators of $\frac{\langle A\rangle}{p}$ with $p=2,3$.}\label{tab-GOr}
    	\begin{tabular}{|c|c|c|c||l|}
    		\hline \hline
    $p$ & $t_1$ & $t_2$ & $t_3$ & A system of generators of $\frac{\langle A\rangle}{p}$  \\
    		\hline
    $2$ & $0$ & $0$ & $1$ & $\langle \frac{a_1}{2},\frac{a_2}{2},a_3\rangle$ \\
    		\hline
    $2$ & $0$ & $1$ & $1$ & $\langle \frac{a_1}{2},a_2,a_3,\frac{a_2+a_3}{2}\rangle$ \\
    		\hline
    $2$ & $1$ & $1$ & $1$ & $\langle a_1,a_2,a_3,\frac{a_1+a_2}{2},\frac{a_1+a_3}{2},\frac{a_2+a_3}{2}\rangle$ \\
    		\hline\hline
    $3$ & $0$ & $0$ & $1$ & $\langle \frac{a_1}{3},\frac{a_2}{3},a_3\rangle$ \\
    		\hline
    $3$ & $0$ & $0$ & $2$ & $\langle \frac{a_1}{3},\frac{a_2}{3},a_3\rangle$ \\
    		\hline
    $3$ & $0$ & $1$ & $1$ & $\langle \frac{a_1}{3},a_2,a_3,\frac{2a_3+a_2}{3},\frac{2a_2+a_3}{3}\rangle$ \\
    		\hline
    $3$ & $0$ & $1$ & $2$ & $\langle \frac{a_1}{3},a_2,a_3,\frac{a_2+a_3}{3}\rangle$ \\
    		\hline
    $3$ & $0$ & $2$ & $2$ & $\langle \frac{a_1}{3},a_2,a_3,\frac{2a_2+a_3}{3},\frac{a_2+2a_3}{3}\rangle$ \\
    		\hline
    $3$ & $1$ & $1$ & $1$ & $\langle a_1,a_2,a_3,\frac{2a_1+a_2}{3},\frac{2a_1+a_3}{3},\frac{2a_2+a_1}{3},\frac{2a_2+a_3}{3},
    \frac{2a_3+a_1}{3},\frac{2a_3+a_2}{3},\frac{a_1+a_2+a_3}{3}\rangle$ \\
    		\hline
    $3$ & $1$ & $1$ & $2$ & $\langle a_1,a_2,a_3,\frac{a_1+a_3}{3},\frac{a_2+a_3}{3},\frac{2a_1+a_2}{3},\frac{2a_2+a_1}{3}\rangle$ \\
    		\hline
    $3$ & $1$ & $2$ & $2$ & $\langle a_1,a_2,a_3,\frac{a_1+a_2}{3},\frac{a_1+a_3}{3},\frac{2a_2+a_3}{3},\frac{2a_3+a_2}{3}\rangle$ \\
    		\hline
    $3$ & $2$ & $2$ & $2$ & $\langle a_1,a_2,a_3,\frac{2a_1+a_2}{3},\frac{2a_1+a_3}{3},\frac{2a_2+a_1}{3},\frac{2a_2+a_3}{3},
    \frac{2a_3+a_1}{3},\frac{2a_3+a_2}{3},\frac{a_1+a_2+a_3}{3}\rangle$ \\
    		\hline
    	\end{tabular}
    \end{table}

Our idea is extended to give the following simple characterizations about $\frac{\langle A\rangle}{p}$.
\begin{thm}\label{generalQuotiA}
Suppose $A=(a_1,a_2,...,a_n)=(pk_1+t_1,pk_2+t_2,...,pk_n+t_n)$ with $\gcd(A)=1$, $p\in\mathbb{Z}^{+}$, $k_i\in \mathbb{N}$, $1\leq t_i\leq p-1$ for $1\leq i\leq n$, $n\geq 2$. Let
$$\mathcal{T}_p=\left\{(x_1,x_2,...,x_n) \mid 0\leq x_1,x_2,...,x_n\leq p-1,\ \ p\ \Big| \ \sum_{i=1}^nx_it_i \ (\neq 0)\right\}.$$
Then a system of generators of the quotient of the numerical semigroup $\langle A\rangle$ by $p$ is given by
$$\frac{\langle A\rangle}{p}=\left\langle a_1,a_2,...,a_n,\frac{1}{p}\sum_{i=1}^nx_ia_i\ \Big|\  (x_1,x_2,...,x_n)\in \mathcal{T}_p\right\rangle.$$
\end{thm}

The paper is organized as follows.
In Section 2, we introduce some preliminary knowledge related to MacMahon's Partition Analysis and the constant term method \cite{Xin04,Xin15}. We calculate $\textrm{RGF}_p(x)$ and obtain a system of generators for  $\frac{\langle 3k_1+1,3k_2+2,3k_3+2\rangle}{3}$ and $\frac{\langle a,a+1\rangle}{a-1}$ to illustrate how to use the method.
In Section 3, we give the proof of Theorem \ref{generalQuotiA}.

\section{Preliminary: MacMahon's Partition Analysis}

In algebraic combinatorics, MacMahon's Partition Analysis \cite{MacMahonCA} is one of the suitable tools for solving counting problems in connection with linear Diophantine equations and inequalities. Such problems can be transformed into the constant term of an Elliott-rational function, which is a rational function
whose denominator is a product of binomials. This type of constant term has been restudied by Andrews et. al. using computer algebra \cite{Andrews2001,Andrews2000,AndrewsPaule}. Algorithms together with implementation have been developed, such as the Omega package \cite{Andrews2001}, the Ell package \cite{Xin04}, and the CTEuclid package \cite{Xin15}. We will work with symbolic data.

We need to introduce some basic definitions and results in \cite{Xin04,Xin15}. The theory relies on the unique series expansion of rational functions, for otherwise, the constant term of a rational function is not well-defined.

We shall work in the field $K=\Q((\lambda))((x))$ of the double Laurent series. In this field, every rational function has a unique Laurent series expansion, 
so that the following definition works.
\begin{dfn}[\cite{Xin04}]\label{dfn-natural}
Suppose an element in $K=\Q((\lambda))((x))$ is written as a formal Laurent series $\sum_{i=-\infty}^{\infty} a_{i}\lambda^{i}$ in $\lambda$,
where $a_{i}$ are elements in $\Q((x))$. Then the constant term operator $\mathrm{CT}_{\lambda}$ acts by
$$\mathop{\mathrm{CT}}\limits_{\lambda}\sum_{i=-\infty}^{\infty}a_{i}\lambda^{i}=a_0.$$
\end{dfn}
This definition extends to $\CT_{\Lambda}$ for a set of variables $\Lambda =\{ \lambda_1,\lambda_2, ...,\lambda_m\}$ in \cite{Xin04}.
Here we only need the case $m=1$.

To use rational functions in $K$, we need to clarify their series expansion. A monomial $M=x^k \lambda^\ell \neq 1$ is said to be \emph{small}, denoted $M<1$,
if $k>0$ or if $k=0$ and $\ell>0$, and is said to be \emph{large}, denoted $M>1$, otherwise. Then we can determine which of
the following two series expansions holds in $K$.
\begin{align*}
\dfrac{1}{1-M}=\left\{
             \begin{array}{lr}
             \sum_{k\geq 0}M^k&\ \text{if}\ M<1; \\
            \dfrac{1}{-M(1-1/M)}=-\sum_{k\geq 0}\dfrac{1}{M^{k+1}}&\ \text{if}\  M>1.
             \end{array}
\right.
\end{align*}

Concerning the series expansion of an Elliott-rational function $E$, we usually write $E$ in its proper form:
\begin{equation}\label{f-Elliott}
E = \frac{L}{
\prod_{j=1}^n (1-M_j)}=L\prod_{j=1}^n \Big(\sum_{k\ge 0} (M_j) ^k\Big),
\end{equation}
where $L$ is a Laurent polynomial and each monomial $M_j$ is small. Note that the proper form of $E$ is not unique.
For instance, $1/(1-x)=(1+x)/(1-x^2)$ are both proper forms.

\subsection{Extract Constant Term}
Consider the $\textrm{RGF}_p(x)$ of a numerical semigroup $\frac{\langle A\rangle}{p}$, where $\langle A\rangle=\langle a_1,a_2,...,a_k\rangle$, $\gcd(A)=1$ and $p\in \mathbb{Z}^{+}$. We introduce a new variable $\lambda$ to replace the linear constraint $pn=c_1a_1+c_2a_2+\cdots +c_ka_k$, so we have
\begin{align}\label{RGFxCT}
\textrm{RGF}_p(x)&=\sum_{n\geq 0}d(pn;a_1,a_2,...,a_k)x^n\notag
\\&=\sum_{n\geq 0, c_i\geq 0}\mathop{\mathrm{CT}}\limits_{\lambda}\lambda^{c_1a_1+c_2a_2+\cdots +c_ka_k-pn}x^n\notag
\\&=\mathop{\mathrm{CT}}\limits_{\lambda}\frac{1}{\left(1-\frac{x}{\lambda^p}\right)
(1-\lambda^{a_1})(1-\lambda^{a_1})\cdots (1-\lambda^{a_k})}.
\end{align}
For the third ``=", we used the sum of a geometric series and the linearity of the $\CT$ operator.
The above expression is a power series in $x$ but the power in $\lambda$ ranges from $-\infty$ to $\infty$.
Thus we have represented $\textrm{RGF}_p(x)$ as the constant term of an Elliott rational function.

\begin{rem}
The Frobenius number of $\frac{\langle A\rangle}{p}$ is the greatest integer $m$ with $\textrm{RGF}_p^{(m)}(0)=0$, i.e.,
$$F\left(\frac{\langle A\rangle}{p}\right)=\max\{n\in \mathbb{N}\mid d(pn;a_1,a_2,...,a_k)=0\}
=\max\left\{m \in\mathbb{N}\mid \textrm{RGF}_p^{(m)}(0)=0\right\}.$$
\end{rem}

To extract the constant term, we use partial fraction decompositions of univariate rational functions, from which the constant term can be read off.
To this end,
we write
\begin{equation}
E=\dfrac{L(\lambda)}{\prod_{i=1}^n(1-u_i\lambda^{a_i})}, \label{NEE}
\end{equation}
where $L(\lambda)$ is a Laurent polynomial, $u_i$ are free of $\lambda$ and $a_i$ are positive integers for all $i$.
Note that we might have $u_i\lambda^{a_i}=x^{-1}\lambda^2>1$, so that \eqref{NEE} is not a proper form.

We need the following three results.
\begin{prop}[\cite{Xin15}]\label{propEPFD}
Suppose that the partial fraction decomposition of $E$ is given by
\begin{equation}
E=P(\lambda)+\dfrac{p(\lambda)}{\lambda^k}+\sum_{i=1}^n\dfrac{A_i(\lambda)}{1-u_i\lambda^{a_i}}, \label{EE}
\end{equation}
where the $u_i$'s are free of $\lambda$, $P(\lambda), p(\lambda)$, and the $A_i(\lambda)$'s are all polynomials, $\deg_p(\lambda)<k$, and $\deg A_{i}(\lambda)<a_i$ for all $i$. Then we have
$$\mathop{\mathrm{CT}}\limits_{\lambda}E=P(0)+\sum_{u_i\lambda^{a_i}<1}A_i(0),$$
where the sum ranges over all $i$ such that $u_i\lambda^{a_i}$ is small in $\Q((\lambda))((x))$.
\end{prop}
The proposition is true because direct series expansion gives
\begin{align*}
\dfrac{A_i(\lambda)}{1-u_i\lambda^{a_i}}=\left\{
             \begin{array}{lr}
             \dfrac{A_i(\lambda)}{1-u_i\lambda^{a_i}} \stackrel{\mathop{\mathrm{CT}}_{\lambda}}{\longrightarrow}A_i(0)&\ \text{if}\ \ u_i\lambda^{a_i}<1; \\
            \dfrac{A_i(\lambda)}{-u_i\lambda^{a_i}(1-\dfrac{1}{u_i\lambda^{a_i}})}
            =\dfrac{\lambda^{-a_i}A_i(\lambda)}{-u_i(1-\dfrac{1}{u_i\lambda^{a_i}})}
            \stackrel{\mathop{\mathrm{CT}}_{\lambda}}{\longrightarrow}0&\ \text{if}\ \ u_i\lambda^{a_i}>1.
             \end{array}
\right.
\end{align*}
For clarity, we have written the rational function in its proper form before applying $\CT_\lambda$.

\begin{thm}[\cite{Xin15}]\label{charAS}
Let $E$ be as in Equation \eqref{EE}. Then $A_s(\lambda)$ is uniquely characterized by
\begin{equation}\label{AS}
             A_s(\lambda)\equiv E\cdot (1-u_s\lambda^{a_s})\ \mod\langle 1-u_s\lambda^{a_s}\rangle,\ \
             deg_{\lambda}A_s<a_s,
\end{equation}
where $\langle 1-u_s\lambda^{a_s}\rangle$ denotes the ideal generated by $1-u_s\lambda^{a_s}$.
\end{thm}

In order to compute $\mathop{\mathrm{CT}}_{\lambda}E$ for $E$ as in Equation \eqref{NEE} in $K$, we need to compute
$$A_s(0):=\mathcal{A}_{1-u_s\lambda^{a_s}}E=\mathcal{A}_{1-(u_s\lambda^{a_s})^{-1}}E,$$
where $A_s(\lambda)$ is characterized by Equation \eqref{AS}. Thus in the new notation, Proposition \ref{propEPFD} reads
\begin{equation}
\mathop{\mathrm{CT}}\limits_{\lambda}E=P(0)
+\sum_{i}\chi(u_i\lambda^{a_i}<1)\mathcal{A}_{1-u_i\lambda^{a_i}}E.\label{EP0}
\end{equation}

\begin{thm}[\cite{Xin15}]\label{LemE}
Let $E$ be as in \eqref{NEE}. If $E$ is proper in $\lambda$, i.e., the degree in the numerator is less than the degree in the denominator, then
\begin{equation}
\mathop{\mathrm{CT}}\limits_{\lambda}E
=\sum_{i=1}^n\chi(u_i\lambda^{a_i}<1)\mathcal{A}_{1-u_i\lambda^{a_i}}E.\label{CE}
\end{equation}
If $E|_{\lambda=0}=\lim_{\lambda\rightarrow 0}E$ exists, then
\begin{equation}
\mathop{\mathrm{CT}}\limits_{\lambda}E
=E|_{\lambda=0}-\sum_{i=1}^n\chi(u_i\lambda^{a_i}>1)\mathcal{A}_{1-u_i\lambda^{a_i}}E.\label{DCE}
\end{equation}
\end{thm}

Equation \eqref{DCE} is a kind of dual of Equation \eqref{CE}. Because of these two formulas, it is convenient to call the denominator factor $1-u_i\lambda^{a_i}$ contributing if $u_i\lambda^{a_i}$ is small and dually contributing
if $u_i\lambda^{a_i}$ is large. Now we also denote
$$\mathop{\mathrm{CT}}\limits_{\lambda}\underline{\dfrac{1}{1-u_s\lambda^{a_s}}}E(1-u_s\lambda^{a_s})=\mathcal{A}_{1-u_s\lambda^{a_s}}E=A_s(0),$$
For this notation, we allow $a_s<0$. One can think that only the single underlined factor of the denominator contributes when taking the constant term in $\lambda$.

\begin{lem}\label{l-000}
If $E$ as in \eqref{NEE} is proper in $\lambda$, i.e., the degree in the numerator is less than the degree in the denominator,  and
$E|_{\lambda=0}=0$, then
$$ \sum_{s=1}^n \mathcal{A}_{1-u_s\lambda^{a_s}}E =0.$$
\end{lem}

\subsection{Two Examples}\label{AnExampleAA+1}

In this subsection, we obtain a system of generators for $\frac{\langle 3k_1+1,3k_2+2,3k_3+2\rangle}{3}$ and $\frac{\langle a, a+1\rangle}{a-1}$ by calculating their representation generating function $\textrm{RGF}_p(x)$, respectively.

\begin{prop}\label{Example3122}
Let $A=(a_1,a_2,a_3)=(3k_1+1,3k_2+2,3k_3+2)$, $k_1,k_2,k_3\in \mathbb{N}$ and $p=3$. Suppose $\gcd(A)=1$.
A system of generators of $\frac{\langle A\rangle}{p}$ is given by
\begin{equation}
  \frac{\langle A\rangle}{3}=\left\langle a_1,a_2,a_3,\frac{a_1+a_2}{3},\frac{a_1+a_3}{3},\frac{2a_2+a_3}{3},\frac{2a_3+a_2}{3}\right\rangle. \label{e-Aq3}
\end{equation}
\end{prop}
\begin{proof}
The right hand side of \eqref{e-Aq3} (RHS) is easily seen to be contained in the left hand side (LHS). To show that {LHS} is contained in RHS, we compute as follows.
By Equation \eqref{RGFxCT}, we have
\begin{align}
\textrm{RGF}_3(x)&=\mathop{\mathrm{CT}}\limits_{\lambda}\frac{1}{(1-\frac{x}{\lambda^{3}})(1-\lambda^{3k_1+1})
(1-\lambda^{3k_2+2})(1-\lambda^{3k_3+2})}\notag
\\&=\mathop{\mathrm{CT}}\limits_{\lambda}\frac{-1}{\underline{(1-\frac{x}{\lambda^{3}})}
(1-\lambda^{3k_1+1})(1-\lambda^{3k_2+2})(1-\lambda^{3k_3+2})}\ \ \text{(by Theorem \ref{LemE})}\notag
\\&=\mathop{\mathrm{CT}}\limits_{\lambda}\frac{-1}{\underline{(1-\frac{x}{\lambda^{3}})}
(1-x^{k_1}\lambda)(1-x^{k_2}\lambda^{2})(1-x^{k_3}\lambda^{2})}\ \ \text{(by Theorem \ref{charAS})}\notag
\\&=\mathop{\mathrm{CT}}\limits_{\lambda}\frac{1}{(1-\frac{x}{\lambda^{3}})
\underline{(1-x^{k_1}\lambda)(1-x^{k_2}\lambda^{2})(1-x^{k_3}\lambda^{2})}}\ \ \text{(by Lemma \ref{l-000})}\notag
\\&=\frac{1}{(1-x^{3k_1+1})(1-x^{k_2-2k_1})(1-x^{k_3-2k_1})}\notag
\\&\ \ +\mathop{\mathrm{CT}}\limits_{\lambda}\frac{1}{\left(1-\frac{x}{\lambda^{3}}\right)
(1-x^{k_1}\lambda)\underline{(1-x^{k_2}\lambda^{2})}(1-x^{k_3}\lambda^{2})}\label{RGFHH3122}
\\&\ \ +\mathop{\mathrm{CT}}\limits_{\lambda}\frac{1}{\left(1-\frac{x}{\lambda^{3}}\right)
(1-x^{k_1}\lambda)(1-x^{k_2}\lambda^{2})\underline{(1-x^{k_3}\lambda^{2})}}.\notag
\end{align}
The second term of \eqref{RGFHH3122} is written as
\begin{align*}
&\mathop{\mathrm{CT}}\limits_{\lambda}\frac{1}{\left(1-\frac{x}{\lambda^{3}}\right)
(1-x^{k_1}\lambda)\underline{(1-x^{k_2}\lambda^{2})}(1-x^{k_3}\lambda^{2})}
\\&=\mathop{\mathrm{CT}}\limits_{\lambda}\frac{1}{\left(1-\frac{x^{k_2+1}}{\lambda}\right)
(1-x^{k_1}\lambda)\underline{(1-x^{k_2}\lambda^{2})}(1-x^{k_3-k_2})}\ \ \text{(by Theorem \ref{charAS})}
\\&=\mathop{\mathrm{CT}}\limits_{\lambda}\frac{-\lambda x^{-k_2-1}}{\left(1-\frac{\lambda}{x^{k_2+1}}\right)
(1-x^{k_1}\lambda)\underline{(1-x^{k_2}\lambda^{2})}(1-x^{k_3-k_2})}
\\&=\mathop{\mathrm{CT}}\limits_{\lambda}\frac{\lambda x^{-k_2-1}}{\underline{\left(1-\frac{\lambda}{x^{k_2+1}}\right)
(1-x^{k_1}\lambda)}(1-x^{k_2}\lambda^{2})(1-x^{k_3-k_2})}\ \ \text{(by Lemma \ref{l-000})}
\\&=\frac{1}{(1-x^{k_1+k_2+1})(1-x^{3k_2+2})(1-x^{k_3-k_2})}+\frac{x^{-k_1-k_2-1}}{(1-x^{-k_1-k_2-1})(1-x^{k_2-2k_1})(1-x^{k_3-k_2})}.
\end{align*}
Similarly, the third term of \eqref{RGFHH3122} is written as
\begin{align*}
&\mathop{\mathrm{CT}}\limits_{\lambda}\frac{1}{\left(1-\frac{x^{k_3+1}}{\lambda}\right)
(1-x^{k_1}\lambda)(1-x^{k_2-k_3})\underline{(1-x^{k_3}\lambda^2)}}\ \ \text{(by Theorem \ref{charAS})}
\\&=\mathop{\mathrm{CT}}\limits_{\lambda}\frac{\lambda x^{-k_3-1}}{\underline{\left(1-\frac{\lambda}{x^{k_3+1}}\right)
(1-x^{k_1}\lambda)}(1-x^{k_2-k_3})(1-x^{k_3}\lambda^2)}\ \ \text{(by Lemma \ref{l-000})}
\\&=\frac{1}{(1-x^{k_1+k_2+1})(1-x^{k_3-k_2})(1-x^{3k_3+2})}+\frac{x^{-k_1-k_3-1}}{(1-x^{-k_1-k_3-1})(1-x^{k_3-k_2})(1-x^{k_3-2k_1})}.
\end{align*}
Therefore, we obtain the representation generating function as follows.
\begin{small}
\begin{align*}
\textrm{RGF}_3(x)&=\frac{1+(x^{k_1+1}+x^{k_2+k_3+2})(x^{k_2}+x^{k_3})+x^{2k_1+2}(x^{2k_2}+x^{2k_3})
+x^{k_1+k_2+k_3}(x^{k_1+2}+x^{k_2+k_3+3})}{(1-x^{3k_1+1})(1-x^{3k_2+2})(1-x^{3k_3+2})}
\\&=\frac{1+x^{\frac{a_1+a_2}{3}}+x^{\frac{a_1+a_3}{3}}+x^{\frac{2a_2+a_3}{3}}+x^{\frac{2a_3+a_2}{3}}
+x^{\frac{2(a_1+a_2)}{3}}+x^{\frac{2(a_1+a_3)}{3}}+x^{\frac{2a_1+a_2+a_3}{3}}
+x^{\frac{a_1+2a_2+2a_3}{3}}}{(1-x^{a_1})(1-x^{a_2})(1-x^{a_3})}.
\end{align*}
\end{small}
By $\frac{2a_1+a_2+a_3}{3}=\frac{a_1+a_2}{3}+\frac{a_1+a_3}{3}$ and $\frac{a_1+2a_2+2a_3}{3}=\frac{a_1+a_2}{3}+\frac{2a_3+a_2}{3}$,
the power of each term in the series expansion of $\textrm{RGF}_3(x)$ is contained in RHS. This completes the proof.
\end{proof}

\begin{prop}\label{HalflineAAAA}
Let $A=(a,a+1)$, $a\geq 3$ and $p=a-1$.
A system of generators of $\frac{\langle a,a+1\rangle}{a-1}$ is given as follows.
\[\frac{\langle a,a+1\rangle}{a-1}=\begin{cases}
\left\langle \dfrac{a+1}{2},\dfrac{a+3}{2},...,a-1,a\right\rangle, &\ \text{if}\ \ a\ \text{is odd}; \\
\left\langle \dfrac{a}{2}+1,\dfrac{a}{2}+2,...,a,a+1\right\rangle, &\ \text{if}\ \ a\ \text{is even}.
\end{cases}\]
\end{prop}
\begin{proof}
Similarly, we only show that LHS is contained in RHS.
By Equation \eqref{RGFxCT}, we have
\begin{align*}
\textrm{RGF}_{a-1}(x)&=\mathop{\mathrm{CT}}\limits_{\lambda}\frac{1}{(1-\frac{x}{\lambda^{p}})(1-\lambda^{a})
(1-\lambda^{a+1})}
\\&=\mathop{\mathrm{CT}}\limits_{\lambda}\frac{-1}{\underline{(1-\frac{x}{\lambda^{a-1}})}
(1-\lambda^{a})(1-\lambda^{a+1})}\ \text{(by Theorem \ref{LemE})}
\\&=\mathop{\mathrm{CT}}\limits_{\lambda}\frac{-1}{\underline{(1-\frac{x}{\lambda^{a-1}})}
(1-x\lambda)(1-x\lambda^{2})}\ \text{(by Theorem \ref{charAS})}
\\&=\mathop{\mathrm{CT}}\limits_{\lambda}\frac{1}{(1-\frac{x}{\lambda^{a-1}})
\underline{(1-x\lambda)(1-x\lambda^2)}}\ \text{(by Lemma \ref{l-000})}
\\&=\frac{1}{(1-x^{a})(1-x^{-1})}
+\mathop{\mathrm{CT}}\limits_{\lambda}\frac{1}{\left(1-\frac{x}{\lambda^{a-1}}\right)
(1-x\lambda)\underline{(1-x\lambda^{2})}}.
\end{align*}
For the second term, we compute according to the parity of $a$. If $a$ is odd, then we have
\begin{align*}
&\mathop{\mathrm{CT}}\limits_{\lambda}\frac{1}{\left(1-\frac{x}{\lambda^{a-1}}\right)
(1-x\lambda)\underline{(1-x\lambda^{2})}}
=\mathop{\mathrm{CT}}\limits_{\lambda}\frac{1}{\left(1-x^{\frac{a+1}{2}}\right)
(1-x\lambda)\underline{(1-x\lambda^{2})}}
\\=&\frac{1}{\left(1-x^{\frac{a+1}{2}}\right)}\left(1-\mathop{\mathrm{CT}}\limits_{\lambda}
\frac{1}{\underline{(1-x\lambda)}(1-x\lambda^2)}\right)
=\frac{1}{\left(1-x^{\frac{a+1}{2}}\right)}\left(1-\frac{1}{1-x^{-1}}\right)
\\=&\frac{1}{(1-x)(1-x^{\frac{a+1}{2}})}.
\end{align*}
Thus, we obtain
\begin{align*}
\textrm{RGF}_{a-1}(x)&=\frac{1-x-x^a+x^{\frac{a+3}{2}}}{(1-x)(1-x^a)\left(1-x^{\frac{a+1}{2}}\right)}
=\frac{1+x^{\frac{a+3}{2}}+x^{\frac{a+5}{2}}+\cdots +x^{a-1}}{(1-x^a)\left(1-x^{\frac{a+1}{2}}\right)};
\end{align*}
If $a$ is even, then we have
\begin{align*}
&\mathop{\mathrm{CT}}\limits_{\lambda}\frac{1}{\left(1-\frac{x}{\lambda^{a-1}}\right)
(1-x\lambda)\underline{(1-x\lambda^{2})}}
=\mathop{\mathrm{CT}}\limits_{\lambda}\frac{1}{\left(1-\frac{x^{\frac{a}{2}}}{\lambda}\right)
(1-x\lambda)\underline{(1-x\lambda^{2})}}
\\=&\mathop{\mathrm{CT}}\limits_{\lambda}\frac{-\frac{\lambda}{x^{a/2}}}
{(1-\frac{\lambda}{x^{a/2}})(1-x\lambda)\underline{(1-x\lambda^2)}}
=\mathop{\mathrm{CT}}\limits_{\lambda}\frac{\frac{\lambda}{x^{a/2}}}
{\underline{(1-\frac{\lambda}{x^{a/2}})(1-x\lambda)}(1-x\lambda^2)}
\\=&\frac{1}{(1-x^{\frac{a+2}{2}})(1-x^{a+1})}+\frac{x}{(1-x^{\frac{a+2}{2}})(1-x)}.
\end{align*}
Thus, we obtain
\begin{align*}
\textrm{RGF}_{a-1}(x)&=\frac{-x}{(1-x)(1-x^{a})}+\frac{1}{(1-x^{\frac{a+2}{2}})(1-x^{a+1})}
+\frac{x}{(1-x^{\frac{a+2}{2}})(1-x)}
\\&=\frac{1-x+(x^{\frac{a+2}{2}}+x^{a+2})(1-x^{\frac{a}{2}-1})}{(1-x)(1-x^a)(1-x^{a+1})}
\\&=\frac{1+(x^{\frac{a+2}{2}}+x^{a+2})(1+x+x^2+\cdots +x^{\frac{a}{2}-2})}{(1-x^a)(1-x^{a+1})}.
\end{align*}
The proposition then follows.
\end{proof}

Note that $\frac{\langle a,a+1\rangle}{a-1}$ is a half-line numerical semigroup (see \cite{M.Bras-Amoros04}). Therefore its Frobenius number is given by
\[F\left(\frac{\langle a,a+1\rangle}{a-1}\right)=\begin{cases}
\dfrac{a-1}{2}, &\ \text{if}\ \ a\ \text{is odd}; \\
\dfrac{a}{2}, &\ \text{if}\ \ a\ \text{is even}.
\end{cases}\]

\section{A System of Generators of $\frac{\langle A\rangle}{p}$}
Let $A=(a_1,a_2)=(pk_1+t_1,pk_2+t_2)$, $0\leq t_i\leq p-1$, $p,k_1,k_2\in\mathbb{Z}^{+}$ and $\gcd(A)=1$. We can compute $\textrm{RGF}_p(x)$ for $p=2,3,4,5$ in a similar way to that in Proposition \ref{Example3122}. Readers are invited to try to compute for themselves. Furthermore, we can obtain a system of generators of $\frac{\langle a_1,a_2\rangle}{p}$. The results, not given here, agree with known results in \cite[Proposition 17]{Rosales2006}.

For $A=(a_1,a_2,a_3)$, we compute similarly to obtain $\textrm{RGF}_{p}(x)$ for $p=2,3$. The corresponding systems of generators are given in Table \ref{tab-GOr}.
This table illustrates a pattern, which we summarize as Theorem \ref{generalQuotiA}. It turns out that the theorem has a simple direct proof as follows.

\begin{proof}[Proof of Theorem \ref{generalQuotiA}]
Let $\overline{B}:=\left\langle a_1,a_2,\ldots,a_n,\frac{1}{p}\sum_{i=1}^nx_ia_i\ \Big|\  (x_1,x_2,\ldots,x_n)\in \mathcal{T}_p\right\rangle$. This is well-defined since $(x_1a_1+x_2a_2+\cdots+x_na_n)/p\in \Z^+$ by definition of $\mathcal{T}_p$. Let $\overline{A}:=\frac{\langle A\rangle}{p}$. Then
the containment $\overline{A}\supseteq \overline{B}$ is obvious, and we need to show that $\overline{A}\subseteq \overline{B}$.

If $x\in \overline{A}=\frac{\langle A\rangle}{p}$, then $xp=y_1a_1+y_2a_2+\cdots+y_na_n$ for some $y_1,y_2,\ldots,y_n\in \mathbb{N}$.
Each $y_i$ is uniquely written as $y_i=m_ip+r_i$ for some $m_i\geq 0$ and $0\leq r_i\leq p-1$. Then we have
$x=m_1a_1+m_2a_2+\cdots+m_na_n+(r_1a_1+r_2a_2+\cdots+r_na_n)/p$ and $p\mid (r_1a_1+r_2a_2+\cdots+r_na_n)$.
Now $(r_1a_1+r_2a_2+\cdots+r_na_n)/p$ is either $0$ or an element in $\left\{(x_1a_1+x_2a_2+\cdots+x_na_n)/p\ |\ (x_1,x_2,...,x_n)\in \mathcal{T}_p\right\}$ by the definition of $\mathcal{T}_p$. In either case, we have $x\in \overline{B}$. Therefore $\overline{A}\subseteq \overline{B}$.
\end{proof}

\begin{cor}
Suppose that $A=(a_1,a_2,a_3)=(pk_1+t_1,pk_2+t_2,pk_3+t_3)$ with $p\in\mathbb{Z}^{+}$, $k_i\in \mathbb{N}$, $1\leq t_i\leq p-1$ for $1\leq i\leq 3$. If $\gcd(A)=1$,
then a system of generators of the quotient of the numerical semigroup $\langle A\rangle$ by $p$ is given by
$$\frac{\langle A\rangle}{p}=\left\langle a_1,a_2,a_3,\frac{1}{p}(x_1a_1+x_2a_2+x_3a_3)\ \Big|\  (x_1,x_2,x_3)\in \mathcal{T}_p\right\rangle,$$
where
$$\mathcal{T}_p=\{(x_1,x_2,x_3) \mid 0\leq x_1,x_2,x_3\leq p-1,\ \ p\mid (t_1x_1+t_2x_2+t_3x_3) ( \neq 0)\}.$$
\end{cor}

We observe that Theorem \ref{generalQuotiA} can be strengthened in the following sense.
Suppose that $A=(a_1,...,a_e,a_{e+1},...,a_n)$, $p\mid a_i$ for $1\leq i\leq e$ and $p\nmid a_j$ for $e+1\leq j\leq n$. Then we have
\begin{align}
\frac{\langle A\rangle}{p}=\left\langle \frac{a_1}{p},\frac{a_2}{p},...,\frac{a_e}{p},\mathcal{L}_p\right\rangle,\label{CharacterT=0}
\end{align}
where $\mathcal{L}_p$ is a system of generators of $\langle a_{e+1},...,a_n\rangle/p$. We only explain {that} LHS is contained in RHS, for the other containment
is trivial. For any $x\in \frac{\langle A\rangle}{p}$, there exists $xp=y_1a_1+y_2a_2+\cdots+y_na_n$ for some $y_1,y_2,\ldots,y_n\in \mathbb{N}$. We have
$x=y_1\frac{a_1}{p}+\cdots +y_e\frac{a_e}{p}+\frac{1}{p}(y_{e+1}a_{e+1}+\cdots +y_na_n)$. Therefore $x\in \left\langle \frac{a_1}{p},\frac{a_2}{p},...,\frac{a_e}{p},\mathcal{L}_p\right\rangle$.

\begin{rem}
Theorem \ref{generalQuotiA} only gives a system of generators of $\frac{\langle A\rangle}{p}$, rather than a system of minimal generators. In \cite{E.Cabanillas}, the author discussed the minimal generators of $\frac{\langle a_1,a_2\rangle}{p}$. In \cite{Rosales2006}, the authors gave a characterization of a system of generators of $\frac{\langle a_1,a_2\rangle}{p}$. In \cite{Alessio19}, the author also gave a depiction of a system of generators of $\frac{\langle A\rangle}{p}$.
\end{rem}

Combining Equation \eqref{CharacterT=0} and Theorem \ref{generalQuotiA}, we reobtain the following result.
\begin{cor}[Corollary 18, \cite{Rosales2006}]
Let $a_1,a_2,k_1,k_2\in \mathbb{Z}^{+}$ and $\gcd(a_1,a_2)=1$. If $a_1=2k_1$, $a_2=2k_2+1$, then
$$\frac{\langle a_1,a_2\rangle}{2}=\left\langle \frac{a_1}{2},a_2\right\rangle.$$
If $a_1=2k_1+1$, $a_2=2k_2+1$, then
$$\frac{\langle a_1,a_2\rangle}{2}=\left\langle a_1,a_2,\frac{a_1+a_2}{2}\right\rangle.$$
\end{cor}

Another consequence of Theorem \ref{generalQuotiA} is the following.
\begin{cor}[Corollary 19, \cite{Rosales2006}]
Let $a_1,a_2,k_1,k_2\in \mathbb{Z}^{+}$ and $\gcd(a_1,a_2)=1$.
If $a_1=3k_1+1$, $a_2=3k_2+1$, (or $a_1=3k_1+2$, $a_2=3k_2+2$,) then
$$\frac{\langle a_1,a_2\rangle}{3}=\left\langle a_1,a_2,\frac{2a_1+a_2}{3},\frac{2a_2+a_1}{3}\right\rangle.$$
If $a_1=3k_1+1$, $a_2=3k_2+2$, then
$$\frac{\langle a_1,a_2\rangle}{3}=\left\langle a_1,a_2,\frac{a_1+a_2}{3}\right\rangle.$$
\end{cor}
\begin{proof}
Apply Theorem \ref{generalQuotiA}. If $(t_1,t_2)=(1,1)$ or $(t_1,t_2)=(2,2)$, then we have $\mathcal{T}_p=\{(1,2),(2,1)\}$. If $(t_1,t_2)=(1,2)$, then we have $\mathcal{T}_p=\{(1,1),(2,2)\}$. This completes the proof.
\end{proof}

\section{Future works}
Let $s\in\mathbb{N}$, $A=(a_1,a_2,...,a_k)$ and $\gcd(A)=1$. T. Komatsu \cite{T.Komatsu} 
introduced \emph{$s$-numerical semigroup} defined as $\langle A;s\rangle=\{n\in \mathbb{N} \mid d(n;a_1,a_2,...,a_k)\geq s+1\} \cup \{0\}$, and 
considered its Frobenius number, which is also called the \emph{$s$-Frobenius number} $F_s(A)$ of $A$. 
In other words, $F_s(A)$ is the largest number $N$ satisfying $d(N;a_1,\dots,a_k)\leq s$. These concepts reduce to the classical one when $s=0$,
and is very hard to compute. These new concepts maybe useful to deepen the idea of this paper.

\noindent
{\small \textbf{Acknowledgements:}
The author would like to thank the referee for helpful comments and suggestions. The author thanks his advisor Guoce Xin for guidance and support. This work is partially supported by the National Natural Science Foundation of China [12071311].

\end{document}